\documentclass[a4paper,reqno]{amsart}

\usepackage[foot]{amsaddr}
\usepackage{graphicx,enumerate,nicefrac,color,bm,cite}
\usepackage{amsmath, amssymb, amstext, amsfonts}
\usepackage{stmaryrd}
\usepackage{comment}
\usepackage{geometry}
\usepackage{subfig}

\usepackage{algorithm,algpseudocode,tabularx}

\makeatletter
\newcommand{\multiline}[1]{%
  \begin{tabularx}{\dimexpr\linewidth-\ALG@thistlm}[t]{@{}X@{}}
    #1
  \end{tabularx}
}
\makeatother

\newcommand{\norm}[1]{\left\|#1\right\|}

\newcommand{\operator}[1]{\mathsf{#1}}
\newcommand{\A}{\operator{A}}

\newcommand{\F}{\operator{F}}

\renewcommand{\H}{\operator{H}}

\renewcommand{\d}{\operator{d}}

\newcommand{\dt}{\,\d t}
\newcommand{\dx}{\,\d\bm x}
\newcommand{\jmp}[1]{\left\llbracket#1\right\rrbracket}

\newcommand{\dprod}[1]{\left<#1\right>_{X^\star\times X}}

\newcommand{\un}[1]{{u}_N^{#1}}

\newcommand{\dissol}[1]{u_{#1}^\star}

\newcommand{\muf}[1]{\mu\left(\left|\nabla #1 \right|^2\right)}
\newcommand{\twon}[2]{\norm{#1}_{L^2(#2)}}

\newcommand{\ceil}[1]{\lceil #1 \rceil}

\newcommand{\abd}{\beta}
\newcommand{\aco}{\alpha}

\newcommand{\Flc}{L_\operator{F}}

\newcommand{\dpa}{\delta}
\newcommand{\Fsm}{\nu}

\newcommand{\tol}{\epsilon_{\textrm{tol}}}

\newcommand{\qref}[1]{q_{\ref{#1}}}
\newcommand{\cref}[1]{C_{\ref{#1}}}

\newtheorem{theorem}{Theorem}[section]
\newtheorem{lemma}[theorem]{Lemma}
\newtheorem{proposition}[theorem]{Proposition} 
\newtheorem{corollary}[theorem]{Corollary} 
\newtheorem{cor}[theorem]{Corollary}

\theoremstyle{definition}

\newtheorem{remark}[theorem]{Remark}

\title[Convergence of adaptive ILG Methods]{On the Convergence of\\ Adaptive Iterative Linearized Galerkin Methods}

\author{Pascal Heid \and Thomas P.~Wihler}
\email{heid@math.unibe.ch \and wihler@math.unibe.ch}
\address{Mathematics Institute, University of Bern, Sidlerstrasse 5, CH-3012 Bern, Switzerland}

\thanks{The authors acknowledge the financial support of the Swiss National Science Foundation (SNF), Grant No. 200021\underline{\space}182524}



\begin{document}

\begin{abstract}
A wide variety of different (fixed-point) iterative methods for the solution of nonlinear equations exists. In this work we will revisit a unified iteration scheme in Hilbert spaces from our previous work~\cite{HeidWihler:18} that covers some prominent procedures (including the Zarantonello, Ka\v{c}anov and Newton iteration methods).
In combination with appropriate discretization methods so-called \emph{(adaptive) iterative linearized Galerkin (ILG) schemes} are obtained. 
The main purpose of this paper is the derivation of an abstract convergence theory for the unified ILG approach (based on general adaptive Galerkin discretization methods) proposed in~\cite{HeidWihler:18}. 
The theoretical results will be tested and compared for the aforementioned three iterative linearization schemes in the context of adaptive finite element discretizations of strongly monotone stationary conservation laws.
\end{abstract}

\maketitle

\section{Introduction}

In this paper we analyze the convergence of adaptive iterative linearized Galerkin (ILG) methods for nonlinear problems with strongly monotone operators. To set the stage, we consider a real Hilbert space $X$ with inner product~$(\cdot,\cdot)_X$ and induced norm denoted by~$\|\cdot\|_X$. Then, given a nonlinear operator~$\F:\,X\to X^\star$, we focus on the equation
\begin{equation}\label{eq:F=0}
u\in X:\qquad\F(u)=0 \quad\text{in }X^\star,
\end{equation} 
where~$X^\star$ denotes the dual space of~$X$. In weak form, this problem reads
\begin{equation}\label{eq:F=0weak}
u\in X:\qquad \dprod{\F(u),v}=0\quad\text{for all } v\in X,
\end{equation}
with~$\dprod{\cdot,\cdot}$ signifying the duality pairing in~$X^\star\times X$. For the purpose of this work, we suppose that $\F$ satisfies the following conditions: \begin{enumerate}[(F1)]
\item The operator $\F$ is \emph{Lipschitz continuous}, i.e.~it exists a constant $\Flc>0$ such that 
\begin{align*}
\left|\dprod{\F(u)-\F(v),w}\right| \leq \Flc \norm{u-v}_X \norm{w}_X,
\end{align*}
for all $u,v,w \in X$.
\item The operator $\F$ is \emph{strongly monotone}, i.e. there is a constant $\Fsm>0$ such that 
\begin{align*}
 \Fsm \norm{u-v}_X^2 \leq \dprod{\F(u)-\F(v),u-v},
\end{align*}
for all $u,v \in X$.
\end{enumerate}
Given the properties (F1) and (F2), the main theorem of strongly monotone operators states that~\eqref{eq:F=0} has a unique solution $u^\star \in X$; see, e.g., \cite[\S3.3]{Necas:86} or \cite[Theorem 25.B]{Zeidler:90}. 

\subsection*{Iterative linearization}
The existence of a solution to the nonlinear equation~\eqref{eq:F=0} can be established in a \emph{constructive} way. This can be accomplished, for instance, by transforming~\eqref{eq:F=0} into an appropriate fixed-point form, which, in turn, induces a potentially convergent fixed-point iteration scheme. To this end, following our approach in~\cite{HeidWihler:18}, for some given~$v\in X$, we consider a linear and invertible \emph{preconditioning operator} $\A[v]:\,X\to X^\star$. Then, applying~$\A[v]^{-1}$ to~\eqref{eq:F=0} leads to the fixed-point equation
\[
u=u-\A[v]^{-1}\F(u).
\]
For any suitable initial guess~$u^0\in X$, the above identity motivates the iteration scheme
\[
u^{n+1}=u^n-\A[u^n]^{-1}\F(u^n),\qquad n\ge 0.
\]
Equivalently, we have
\begin{equation}\label{eq:fp0}
u^{n+1}\in X:\qquad \A[u^n]u^{n+1}=\A[u^n]u^n-\F(u^n),\qquad n\ge 0.
\end{equation}
For given~$u^n\in X$, we emphasize that the above problem of solving for~$u^{n+1}$ is \emph{linear}; consequently, we call~\eqref{eq:fp0} an \emph{iterative linearization scheme} for~\eqref{eq:F=0}. Letting
\begin{equation}\label{eq:f(u)}
f:X \to X^\star,\qquad f(u):=\A[u]u-\F(u),
\end{equation}
we may write
\begin{align}\label{eq:fp1}
\A[u^n]u^{n+1}=f(u^n), \qquad n \geq 0.
\end{align}
In order to discuss the weak form of~\eqref{eq:fp1}, for a prescribed~$u\in X$, we introduce the bilinear form
\begin{equation}\label{eq:Aweak}
a(u;v,w):=\dprod{\A[u]v,w}, \qquad v,w\in X.
\end{equation}
Then, based on~$u^n\in X$, the solution~$u^{n+1}\in X$ of~\eqref{eq:fp1} can be obtained from the weak formulation
\begin{equation}\label{eq:itweak}
a(u^n;u^{n+1},w)=\dprod{f(u^n),w}\qquad \forall w\in X.
\end{equation}
Throughout this paper, for any~$u\in X$, we assume that the bilinear form~$a(u;\cdot,\cdot)$ is uniformly coercive and bounded. The latter two assumptions refer to the fact that there are two constants~$\alpha,\beta>0$ independent of $u \in X$, such that
\begin{equation}\label{eq:coercive}
a(u;v,v) \geq \aco \|v\|_X^2 \qquad\forall v \in X,
\end{equation}
and
\begin{equation}\label{eq:continuity}
a(u;v,w) \leq \abd \norm{v}_X \norm{w}_X \qquad\forall  v,w \in X,
\end{equation}
respectively. In particular, owing to the Lax-Milgram Theorem, these properties imply the well-posedness of the solution~$u^{n+1}\in X$ of the linear equation~\eqref{eq:itweak}, for any given~$u^n\in X$.

Let us briefly review some prominent procedures that can be cast into the framework of the linearized fixed-point iteration~\eqref{eq:itweak}: For instance, we point to the Zarantonello iteration given by
\begin{align}\label{eq:zarantonelloit}
 (u^{n+1},\cdot)_X=(u^{n},\cdot)_X-\dpa \dprod{\F(u^{n}),\cdot},\qquad n\ge 0,
\end{align}
with~$\delta>0$ being a sufficiently small parameter; cf.~Zarantonello's original report~\cite{Zarantonello:60}, or the monographs~\cite[\S3.3]{Necas:86} and~\cite[\S25.4]{Zeidler:90}. A further example is the Ka\v{c}anov scheme which reads
\begin{align} \label{eq:kacanovstrong}
\A[u^{n}]u^{n+1}=g,\qquad n\ge 0,
\end{align}
in the special case that $g=\A[u]u-\F(u)$ is independent of~$u$. Finally, we mention the (damped) Newton method which is defined by
\begin{align} \label{eq:newtonstrong}
 \F'(u^{n})u^{n+1}=\F'(u^{n})u^{n}-\dpa(u^n) \F(u^{n}),\qquad n\ge 0,
\end{align}
for a damping parameter~$\delta(u^n)>0$. Here~$\F'$ signifies the G\^{a}teaux derivative of~$\F$ (provided that it exists). For any of the above three iterative procedures, we emphasize that convergence to the unique solution of~\eqref{eq:F=0} can be guaranteed under suitable conditions; see our previous work~\cite{HeidWihler:18} for details. 

\subsection*{The ILG approach}
Consider a finite dimensional subspace $X_N\subset X$. Then, the Galerkin approximation of~\eqref{eq:F=0weak} in~$X_N$ reads as follows:
\begin{equation}\label{eq:discreteproblem}
\dissol{N} \in X_N:\qquad\dprod{\F(\dissol{N}),v}=0\qquad\forall v \in X_N.
\end{equation}
We note that \eqref{eq:discreteproblem} has a unique solution~$\dissol{N}\in X_N$ since the restriction $\F|_{X_N}$ still satisfies the conditions~(F1) and~(F2) above. The \emph{iterative linearized Galerkin (ILG)} approach is based on discretizing the iteration scheme~\eqref{eq:itweak}.
Specifically, a Galerkin approximation $\un{n+1} \in X_N$ of $\dissol{N}$, based on a prescribed initial guess $\un{0} \in X_N$, is obtained by solving iteratively the linear discrete problem
\begin{equation} \label{eq:lindisproblem}
 \un{n+1} \in X_N: \qquad a(\un{n};\un{n+1},v)=\dprod{f(\un{n}),v} \qquad \forall v \in X_N,
\end{equation}
for~$n\ge 0$. For the resulting sequence~$\{u_N^n\}_{n\ge 0}\subset X_N$ of discrete solutions it is possible, based on elliptic reconstruction techniques (cf., e.g.,~\cite{LakkisMakridakis:06,MakridakisNochetto:03}), to obtain general (abstract) \emph{a posteriori} estimates for the difference to the exact solution, $u^\star\in X$, of~\eqref{eq:F=0}, i.e.~for~$\|u^\star-u_N^{n+1}\|_X$, $n\ge 0$, see~\cite[\S3]{HeidWihler:18}. Based on such \emph{a posteriori} error estimators, an \emph{adaptive ILG} algorithm that exploits an efficient interplay of the iterative linearization scheme~\eqref{eq:lindisproblem} and automatic Galerkin space enrichments was proposed in~\cite[\S4]{HeidWihler:18}; see also~\cite{CongreveWihler:17}. We refer to some related works in the context of (inexact) Newton schemes~\cite{AmreinWihler:14,AmreinWihler:15,ErnVohralik:13,El-AlaouiErnVohralik:11}, or of the Ka\v{c}anov iteration~\cite{BernardiDakroubMansourSayah:15,GarauMorinZuppa:11}.

\subsection*{Goal of this paper}

The convergence of an adaptive Ka\v{c}anov algorithm, which is based on a finite element discretization, for the numerical solution of quasi-linear elliptic partial differential equations has been studied in~\cite{GarauMorinZuppa:11}. 
Furthermore, more recently, the authors of~\cite{GantnerHaberlPraetoriusStiftner:17}  have proposed and analyzed an adaptive algorithm for the numerical solution of~\eqref{eq:F=0} within the specific context of a finite element discretization of the Zarantonello iteration~\eqref{eq:zarantonelloit}. The latter paper includes an analysis of the convergence rate which is related to the work~\cite{CarstensenFeischlPagePraetorius:14} on optimal convergence for adaptive finite element methods within a more general abstract framework. The purpose of the current paper is to generalize the adaptive ILG algorithm from~\cite{GantnerHaberlPraetoriusStiftner:17} to the framework of the \emph{unified iterative linearization scheme}~\eqref{eq:fp1}; furthermore, \emph{arbitrary (conforming) Galerkin discretizations} will be considered. In order to provide a convergence analysis for the ILG scheme~\eqref{eq:lindisproblem} within this general abstract setting, we will follow along the lines of~\cite{GantnerHaberlPraetoriusStiftner:17}, however, we emphasize that some significant modifications in the analysis are required. Indeed, whilst the theory in~\cite{GantnerHaberlPraetoriusStiftner:17} relies on a contraction argument for the Zarantonello iteration, this favourable property is not available for the general iterative linearization scheme~\eqref{eq:fp1}. To address this difficulty, we derive a contraction-like property instead. This observation will then suffice to establish the convergence of the adaptive ILG scheme, and to (uniformly) bound the number of linearization steps on each (fixed) Galerkin space similar to~\cite{GantnerHaberlPraetoriusStiftner:17}; we note that the latter property constitutes a crucial ingredient with regards to the (linear) computational complexity of adaptive iterative linearized finite element schemes.

\subsection*{Outline}

Section~\ref{sec:itlin} contains a convergence analysis of the unified iteration scheme~\eqref{eq:fp1}. On that account we will encounter a contraction-like property, which is key for the subsequent analysis of the convergence rate of the adaptive ILG algorithm in Section~\ref{sec:abstractILG}. Here, in addition, a (uniform) bound of the iterative linearization steps on each discrete space will be shown. In Section~\ref{sec:examples}, we will test our ILG algorithm in the context of finite element discretizations of stationary conservation laws. Finally, we add a few concluding remarks in Section~\ref{sec:conclude}.

\section{Iterative linearization} \label{sec:itlin}

In this first section we will address the convergence of the linearized iteration~\eqref{eq:fp1}. We begin with the following \emph{a posteriori} error estimate.

\begin{lemma} \label{lem:discreteerrorestimate} 
Consider the sequence~$\{u^n\}_{n\ge 0}\subset X$ generated by the iteration~\eqref{eq:fp1}. If $\F$ satisfies~{\rm (F1)--(F2)}, and $a(u;\cdot,\cdot)$, for $u \in X$, fulfils \eqref{eq:coercive}--\eqref{eq:continuity}, then it holds the bound
\begin{equation} \label{eq:discreteerrorestimate} 
 \norm{u^\star-u^{n}}_X \leq \cref{eq:discreteerrorestimate}\norm{u^{n}-u^{n-1}}_X,
 \qquad\text{with}\quad \cref{eq:discreteerrorestimate}:=1+ \nicefrac{\beta}{\Fsm},
\end{equation}
for any~$n\ge 1$.
\end{lemma}

\begin{proof}
By invoking (F2), and since $u^\star$ is the (unique) solution of \eqref{eq:F=0}, for~$n\ge 1$, we find that
 \begin{align*}
  \Fsm \norm{u^\star-u^{n-1}}_X^2 &\leq \dprod{\F(u^\star)-\F(u^{n-1}),u^\star-u^{n-1}} = \dprod{\F(u^{n-1}),u^{n-1}-u^\star}.
 \end{align*}
Employing \eqref{eq:f(u)}, \eqref{eq:itweak}, and \eqref{eq:continuity}, we further get
\begin{align*}
 \Fsm \norm{u^\star-u^{n-1}}_X^2 \leq a(u^{n-1};u^{n-1}-u^{n},u^{n-1}-u^\star) \leq \beta \norm{u^{n}-u^{n-1}}_X \norm{u^{n-1}-u^\star}_X,
\end{align*}
and thus
\begin{align*} 
\norm{u^\star-u^{n-1}}_X \leq \beta \Fsm^{-1} \norm{u^{n}-u^{n-1}}_X.
\end{align*}
By the triangle inequality, this leads to 
\begin{align*}
 \norm{u^\star-u^{n}}_X \leq \norm{u^{n}-u^{n-1}}_X+\norm{u^\star-u^{n-1}}_X \leq \cref{eq:discreteerrorestimate}\norm{u^{n}-u^{n-1}}_X, 
\end{align*}
which completes the proof.
\end{proof}

\begin{remark}\label{rem:XN}
We note that the above result equally holds if~\eqref{eq:F=0weak} and~\eqref{eq:fp1} are restricted to any closed subspace of $X$. 
\end{remark}

\subsection{Potentials} 

In addition to~(F1) and (F2), let us make a further assumption on the (nonlinear) operator~$\F$ from~\eqref{eq:F=0}, which will play an essential role in the ensuing analysis.

\begin{itemize}
\item[(F3)]   The operator $\F$ possesses a potential, i.e.~it exists a G\^{a}teaux differentiable functional $\H:X \to \mathbb{R}$ such that $\H'=\F$.
\end{itemize}

We notice the following relation between the norm~$\|\cdot\|_X$ and the potential~$\H$.

\begin{lemma} \label{lemma:energydifference}
 Suppose that the operator $\F$ satisfies {\rm (F1)--(F3)}, and denote by $u^\star\in X$ the unique solution of \eqref{eq:F=0}. Then, we have the estimate
 \begin{equation} \label{eq:energydifference}
  \frac{\Fsm}{2} \norm{u^\star - u}_X^2 \leq \H(u)-\H(u^\star) \leq \frac{\Flc}{2} \norm{u^\star-u}^2_X \qquad \forall u \in X.
 \end{equation}
In particular, $\H$ takes its minimum at $u^\star$. 
\end{lemma}

\begin{proof}
For fixed $u \in X$, define the real-valued function $\varphi(t):=\H(u^\star+t(u-u^\star))$, for $t \in [0,1]$. Taking the derivative leads to
 \begin{align*}
  \varphi'(t)=\dprod{\H'(u^\star+t(u-u^\star)),u-u^\star}=\dprod{\F(u^\star+t(u-u^\star)),u-u^\star}.
 \end{align*}
By invoking the fundamental theorem of calculus and implementing~\eqref{eq:F=0}, this yields
\begin{align*}
 \H(u)-\H(u^\star)&=\int_0^1 \dprod{\F(u^\star+t(u-u^\star)),u-u^\star} \dt \\
 &= \int_0^1 \dprod{\F(u^\star+t(u-u^\star))-\F(u^\star),u-u^\star} \dt.
\end{align*}
Applying the assumptions (F1) and (F2) we can bound the integrand from above and below, respectively. Indeed, the strong monotonicity (F2) implies that
\begin{align*}
 \H(u)-\H(u^\star)&=\int_0^1 t^{-1}\dprod{\F(u^\star+t(u-u^\star))-\F(u^\star),t(u-u^\star)} \dt 
  \geq \int_0^1 \Fsm t \norm{u^\star-u}_X^2 \dt,
\end{align*}
and therefore,
\begin{equation*}
 \H(u)-\H(u^\star) \geq \frac{\Fsm}{2} \norm{u^\star-u}_X^2.
\end{equation*}
Likewise, by invoking (F1) instead of (F2), we find that
\begin{equation*}
 \H(u)-\H(u^\star) \leq \frac{\Flc}{2} \norm{u^\star-u}_X^2.
\end{equation*}
Combining the above bounds leads to~\eqref{eq:energydifference}.
\end{proof}

\subsection{Contractivity}

In order to state and prove the main result of this section, see Theorem~\ref{thm:it} below, we impose a \emph{monotonicity condition} on the sequence generated by the iterative linearization scheme~\eqref{eq:fp1}.

\begin{itemize}
\item[(F4)] There is a constant $C_\H>0$ such that the sequence defined by \eqref{eq:fp1} fulfils the bound
 \begin{align} \label{eq:Hconstant}
  \H(u^{n-1})-\H(u^{n}) \geq C_\H \norm{u^{n}-u^{n-1}}_X^2 \qquad \forall n \geq 1,
 \end{align}
where $\H$ is the potential of $\F$ introduced in (F3).
\end{itemize}

\begin{proposition} \label{prop:contractionlike}
Suppose that {\rm (F1)--(F4)} are satisfied. Furthermore, let the bilinear form $a(\cdot;\cdot,\cdot)$ from~\eqref{eq:itweak} be coercive and bounded, cf.~\eqref{eq:coercive} and~\eqref{eq:continuity}, respectively. Then, for any $k\ge 1$, the sequence~$\{u^n\}_{n\ge 0}$ from~\eqref{eq:fp1} satisfies the estimate
\begin{equation}\label{eq:sumbound0}
\sum_{j=k+1}^{\infty} \norm{u^{j}-u^{j-1}}_X^2 \leq \cref{eq:sumbound} \norm{u^{k}-u^{k-1}}_X^2, 
\end{equation}
 with
 \begin{equation}\label{eq:sumbound}
 \cref{eq:sumbound}:=\frac{\Flc\cref{eq:discreteerrorestimate}^2}{2C_\H}.
 \end{equation}
Moreover, the contraction-like property
\begin{align} \label{eq:ndiff1diff}
\norm{u^{n}-u^{n-1}}_X^2 \leq \cref{eq:sumbound}\left(1+\cref{eq:sumbound}^{-1}\right)^{2-n} \norm{u^{1}-u^{0}}_X^2
\end{align}
holds true for any~$n\ge 2$.
\end{proposition}
 
Before turning to the proof of the above proposition, we establish an auxiliary result.
 
\begin{lemma}\label{lem:seq}
Consider a sequence~$\{a_j\}_{j=1}^\infty\subset[0,\infty)$ which satisfies the estimate
\begin{equation}\label{eq:seq}
c\sum_{j=k+1}^\infty a_j\le a_k\qquad\forall k\ge 1,
\end{equation}
for some constant~$c>0$. Then, it holds the bound $a_j\le c^{-1}(1+c)^{2-j}a_1$,
for any~$j\ge 2$.
\end{lemma}

\begin{proof}
Let us define the sequence $b_k:=\sum_{j=k}^\infty a_j$, $k\ge 1$.
Using~\eqref{eq:seq}, we note that 
\[
b_{k}
=a_k+\sum_{j=k+1}^\infty a_j
=a_{k}+b_{k+1}\ge (c+1)b_{k+1},
\]
for all~$k\ge1$. By induction, this implies that~$b_2\ge(c+1)^{k-2}b_{k}$ for any~$k\ge 2$. Therefore, we infer that
\[
a_1\ge cb_2\ge c(c+1)^{k-2}b_{k}\ge c(c+1)^{k-2}a_k\qquad\forall k\ge 2.
\]
Rearranging terms completes the proof.
\end{proof}
 
\begin{proof}[Proof of Proposition~\ref{prop:contractionlike}]
Let $n>k\ge 1$ be arbitrary. Then, we note the telescope sum
\begin{align*}
 \H(u^k)-\H(u^n)= \sum_{j=k}^{n-1}\left(\H(u^j)-\H(u^{j+1})\right).
\end{align*}
Thus, by virtue of~\eqref{eq:Hconstant}, we infer that
\begin{align} \label{eq:sumhkn}
\H(u^{k})-\H(u^n) \geq C_\H \sum_{j=k}^{n-1} \norm{u^{j+1}-u^{j}}_X^2.
\end{align}
We aim to bound the left-hand side. To this end, we employ Lemma~\ref{lemma:energydifference}, which implies that
\begin{align*}
\H(u^{k})-\H(u^n) \leq \H(u^{k})-\H(u^\star) \leq \frac{\Flc}{2} \norm{u^\star-u^{k}}_X^2.
\end{align*}
This, together with Lemma~\ref{lem:discreteerrorestimate}, leads to
\begin{align} \label{eq:hkn}
 \H(u^{k})-\H(u^n) \leq \frac{\Flc}{2} \cref{eq:discreteerrorestimate}^2 \norm{u^{k}-u^{k-1}}_X^2.
\end{align}
Combining \eqref{eq:sumhkn} and \eqref{eq:hkn} yields
\[
 \sum_{j=k}^{n-1} \norm{u^{j+1}-u^{j}}_X^2 \leq \cref{eq:sumbound}\norm{u^{k}-u^{k-1}}_X^2.
\]
Letting~$n\to\infty$, we obtain~\eqref{eq:sumbound0}.
Moreover, upon setting~$c:=\cref{eq:sumbound}^{-1}$ and~$a_j:=\|u^{j}-u^{j-1}\|^2_X$, $j\ge 1$, the bound~\eqref{eq:sumbound0} takes the form~\eqref{eq:seq}. Hence, applying Lemma~\ref{lem:seq}, we deduce~\eqref{eq:ndiff1diff}.
\end{proof}

From~\eqref{eq:ndiff1diff} we immediately obtain the following result.

\begin{corollary}\label{cor:null}
Under the assumptions of Proposition~\ref{prop:contractionlike}, it follows that $\norm{u^{n}-u^{n-1}}_X$ is a null sequence as~$n\to\infty$.
\end{corollary}

\subsection{Convergence}

We are now ready to state and prove the main result of this section.

\begin{theorem}\label{thm:it}
Suppose that {\rm (F1)--(F4)} as well as \eqref{eq:coercive} and \eqref{eq:continuity} hold true. Then, the sequence $\{u^n\}_{n \geq 0}$ obtained from the iterative linearization procedure~\eqref{eq:fp1} converges to the unique solution $u^\star \in X$ of \eqref{eq:F=0}.
\end{theorem}

\begin{proof}
Combining Lemma~\ref{lem:discreteerrorestimate} and Corollary~\ref{cor:null} directly implies the convergence of the linearized iteration scheme~\eqref{eq:fp1}.
\end{proof}

\begin{remark}
In the proof of Theorem~\ref{thm:it} the application of Lemma~\ref{lem:discreteerrorestimate} can be replaced by using ~\cite[Proposition~2.1]{HeidWihler:18} instead. We note that the latter result does not require property~(F1) to hold. Indeed, assume that~{\rm (F2)}, \eqref{eq:coercive} and~\eqref{eq:continuity} are satisfied, and that $u \mapsto a(u;u,\cdot)$ and $u \mapsto f(u)$ are continuous mappings from $X$ into its dual space~$X^\star$ with respect to the weak topology on $X^\star$. Then, if the sequence $\{u^{n}\}_{n \geq 0}$ defined by \eqref{eq:fp1} satisfies~$\|u^{n}-u^{n-1}\|_X\to0$ as~$n\to\infty$, it converges to the unique solution $u^\star\in X$ of~\eqref{eq:F=0}.
\end{remark}

\subsection{Some remarks on condition~(F4)}

Suppose that the assumptions (F1)--(F3) are satisfied, and consider the sequence $\{u^n\}_{n \geq 0}$ generated by the iteration~\eqref{eq:fp1}. Analogously as in the proof of Lemma~\ref{lemma:energydifference}, for fixed $n \geq 1$, we define the real-valued function $\varphi(t):=\H(u^{n-1}+t(u^{n}-u^{n-1}))$, for $t \in [0,1]$. Then, it holds the identity
\begin{align*}
 \H(u^{n-1})-\H(u^{n})&=  - \int_0^1 \dprod{\F(u^{n-1}+t(u^{n}-u^{n-1}))-\F(u^{n-1}),u^{n}-u^{n-1}} \dt\\
  & \quad - \dprod{\F(u^{n-1}),u^{n}-u^{n-1}}.
\end{align*}
Using~\eqref{eq:f(u)} and~\eqref{eq:itweak}, we note that
\begin{align*}
- \dprod{\F(u^{n-1}),u^{n}-u^{n-1}}
&= a(u^{n-1};u^{n}-u^{n-1},u^{n}-u^{n-1}).
\end{align*}
Hence,
\begin{align*}
 \H(u^{n-1})-\H(u^{n})
 &=
 - \int_0^1 \dprod{\F(u^{n-1}+t(u^{n}-u^{n-1}))-\F(u^{n-1}),u^{n}-u^{n-1}} \dt\\
&\quad +a(u^{n-1};u^{n}-u^{n-1},u^{n}-u^{n-1}).
\end{align*}
Consequently, if the bilinear form $a(u;\cdot,\cdot)$, for any given $u \in X$, is uniformly coercive with constant $\alpha > \nicefrac{\Flc}{2}$, cf.~\eqref{eq:coercive}, where $\Flc$ refers to the Lipschitz constant occurring in~(F1), then we obtain that
\begin{align*}
\H(u^{n-1})-\H(u^{n})
& \geq \alpha \norm{u^{n}-u^{n-1}}_X^2 - \int_0^1 t \Flc \norm{u^{n}-u^{n-1}}^2_X \dt 
= \left(\alpha-\frac{\Flc}{2}\right) \norm{u^{n}-u^{n-1}}_X^2, 
\end{align*}
i.e.~\eqref{eq:Hconstant} is satisfied with $C_\H=\alpha - \nicefrac{\Flc}{2}>0$. 
  
 \begin{proposition} \label{prop:f4}
  If $\F$ satisfies {\rm (F1)--(F3)}, and the bilinear form $a(\cdot;\cdot,\cdot)$ from the unified iteration scheme \eqref{eq:fp1} is coercive with coercivity constant~$\alpha > \nicefrac{\Flc}{2}$, cf.~\eqref{eq:coercive}, then {\rm (F4)} holds true.
\end{proposition}

\begin{remark}
For the Zarantonello iteration scheme~\eqref{eq:zarantonelloit} we note that $a(u;v,w)=\delta^{-1}(v,w)_{X}$, for $u,v,w \in X$, in~\eqref{eq:Aweak}. Then, we have that
\begin{align*}
 a(u;v,v)=\frac{1}{\delta}\norm{v}_X^2 \qquad \forall u,v \in X,
\end{align*}
which, upon using Proposition~\ref{prop:f4}, shows that~(F4) is satisfied for any $\delta \in (0,\nicefrac{2}{\Flc})$. Under suitable assumptions, a similar observation can be made for the Newton method~\eqref{eq:newtonstrong} provided that the damping parameter~$\delta(u^n)$ is chosen sufficiently small; cf.~\cite[Theorem~2.6]{HeidWihler:18}.
\end{remark}

\begin{remark} 
The above Proposition~\ref{prop:f4} delivers a sufficient condition for (F4). We note, however, that it is not necessary. In particular, if the coercivity constant~$\alpha$ in~\eqref{eq:coercive} is much smaller than the Lipschitz constant $\Flc$ from~(F1), then the bound on~$\alpha$ in Proposition~\ref{prop:f4} is violated. Nonetheless, in that case, we can still satisfy~\eqref{eq:Hconstant} by imposing alternative assumptions; cf., e.g.,~(K2) in \cite{HeidWihler:18}. 
\end{remark}

\section{Adaptive ILG Discretizations} \label{sec:abstractILG} 

In this section, following the recent approach~\cite{GantnerHaberlPraetoriusStiftner:17}, we will present an adaptive ILG algorithm that exploits an interplay of the unified iterative linearization procedure~\eqref{eq:fp1} and abstract adaptive Galerkin discretizations thereof, cf.~\eqref{eq:lindisproblem}. Moreover, we will establish the (linear) convergence of the resulting sequence of approximations to the unique solution of \eqref{eq:F=0}, and comment on the uniform boundedness of the iterative linearization steps on each discrete space. We proceed along the ideas of~\cite[\S4 and \S5]{GantnerHaberlPraetoriusStiftner:17}, and generalize those results to the abstract framework considered in the current paper. Throughout this section, we will assume that any iterative linearization is of the form~\eqref{eq:itweak}, with \eqref{eq:coercive} and \eqref{eq:continuity} being satisfied.

\subsection{Abstract error estimators}

We generalize the assumptions on the finite element refinement indicator from~\cite[\S4]{GantnerHaberlPraetoriusStiftner:17}. Let us consider a sequence of hierarchical finite dimensional Galerkin subspaces $\{X_N\}_{N\ge 0}\subset X$, i.e.
\begin{equation}\label{eq:nested}
X_0\subset X_1\subset X_2\subset \cdots\subset X.
\end{equation}
Suppose that, for any $N\ge0$, there is a computable \emph{error estimator} 
\begin{equation}\label{eq:eta}
\eta_N:\,X_N\to[0,\infty),
\end{equation}
which satisfies the following two properties: 
\begin{enumerate}
\item[(A1)] For all $u,v \in X_N$ it holds that
\begin{equation} \label{eq:a1}
|\eta_N(u)-\eta_N(v)| \leq \cref{eq:a1} \norm{u-v}_X.
\end{equation}
\item[(A2)] The error of the discrete solution~$\dissol{N}\in X_N$ from~\eqref{eq:discreteproblem} is controlled by the \emph{a posteriori} error bound
\begin{equation} \label{eq:a2}
\norm{u^\star-\dissol{N}}_X \leq \cref{eq:a2} \eta_N(\dissol{N}),
\end{equation}
where~$u^\star\in X$ is the exact solution of~\eqref{eq:F=0}.
\end{enumerate}
Here, $\cref{eq:a1}, \cref{eq:a2}\ge1$ are two constants. 


The following result shows that the two estimators for $\un{n}$ and $\dissol{N}$ are equivalent once the linearization error is small enough.

\begin{lemma} \label{lem:equivalentestimators} 
Suppose that $\F$ satisfies {\rm (F1)--(F2)}, and that the \emph{a posteriori} estimator fulfils {\rm (A1)}. Furthermore, for some~$n\ge 1$, assume that
\begin{equation}\label{eq:linerr}
\norm{\un{n}-\un{n-1}}_X \leq \lambda \eta_N(\un{n}),
\end{equation}
with a constant~$\lambda\in(0, \cref{eq:clambda}^{-1})$, where
\begin{align} \label{eq:clambda}
\cref{eq:clambda}:=\cref{eq:discreteerrorestimate}\cref{eq:a1}.                                                                                                                                                                                                    \end{align}
Then, we have that
\begin{equation}\label{eq:unusbound}
\norm{\dissol{N}-\un{n}}_X \leq \lambda\cref{eq:discreteerrorestimate}\min\left\{\eta_N(\un{n}),(1-\lambda \cref{eq:clambda})^{-1} \eta_N(\dissol{N})\right\}.
\end{equation}
Moreover, the two error estimators~$\eta_N(\un{n})$ and~$\eta_N(\dissol{N})$ are equivalent in the sense that
\begin{equation}\label{eq:eq}
 (1-\lambda \cref{eq:clambda}) \eta_N(\un{n}) \leq \eta_N(\dissol{N}) \leq (1+\lambda \cref{eq:clambda}) \eta_N(\un{n}).
\end{equation}
\end{lemma}

\begin{proof}
Owing to Lemma~\ref{lem:discreteerrorestimate} and Remark~\ref{rem:XN}, and due to~\eqref{eq:linerr}, it holds that
\begin{align} \label{eq:firstinequality}
 \norm{\dissol{N}-\un{n}}_X \leq \cref{eq:discreteerrorestimate} \norm{\un{n}-\un{n-1}}_X \leq \lambda\cref{eq:discreteerrorestimate} \eta_N(\un{n}).
\end{align}
Invoking the Lipschitz continuity (A1), we obtain
\begin{align*}
 \norm{\dissol{N}-\un{n}}_X \leq \lambda\cref{eq:discreteerrorestimate}\left(\eta_N(\dissol{N})+\cref{eq:a1}\norm{\dissol{N}-\un{n}}_X\right).
\end{align*}
Since $\lambda < \cref{eq:clambda}^{-1}$ we have that $\lambda\cref{eq:discreteerrorestimate}\cref{eq:a1} =\lambda \cref{eq:clambda}<1$. Hence, manipulating the above inequality yields
\begin{align} \label{eq:secondinequality}
\norm{\dissol{N}-\un{n}}_X \leq \frac{\lambda\cref{eq:discreteerrorestimate}}{1-\lambda \cref{eq:clambda}} \eta_N(\dissol{N}). 
\end{align}
Combining the two inequalities \eqref{eq:firstinequality} and \eqref{eq:secondinequality} gives the bound \eqref{eq:unusbound}.  Moreover,
applying (A1), as before, and using~\eqref{eq:firstinequality}, we infer that
\[
 \eta_N(\dissol{N}) 
 \leq \eta_N(\un{n})+ \cref{eq:a1} \norm{\dissol{N}-\un{n}}_X 
 \leq (1+\lambda \cref{eq:clambda}) \eta_N(\un{n}).
\]
Similarly, employing~\eqref{eq:secondinequality}, it follows that
\begin{align*}
 \eta_N(\un{n}) 
 \leq \eta_N(\dissol{N})+ \cref{eq:a1} \norm{\dissol{N}-\un{n}}_X 
 \leq \left(1+ \frac{\lambda \cref{eq:clambda}}{1-\lambda \cref{eq:clambda}}\right)\eta_N(\un{\star})=\frac{1}{1-\lambda \cref{eq:clambda}} \eta_N(\dissol{N}).
\end{align*}
This completes the argument.
\end{proof}

\subsection{Adaptive ILG algorithm}
We focus on the adaptive algorithm from~\cite{GantnerHaberlPraetoriusStiftner:17}, which was studied in the context of finite element discretizations of the Zarantonello iteration~\eqref{eq:zarantonelloit}. It is closely related to the general adaptive ILG scheme in~\cite{HeidWihler:18}. The key idea is the same in both algorithms: On a given Galerkin space, we iterate the linearization scheme~\eqref{eq:lindisproblem} as long as the linearization error dominates. Once the ratio of the linearization error and the \emph{a posteriori} error bound is sufficiently small, we enrich the Galerkin space in a suitable way. 

\begin{algorithm}
\caption{Adaptive ILG algorithm}
\label{alg:praetal}
\begin{algorithmic}[1]
\State Prescribe a tolerance~$\tol>0$, and an adaptivity parameter $\lambda >0$. Moreover, set~$N:=0$ and~$n:=0$. Start with an initial Galerkin space $X_0\subset X$, and an arbitrary initial guess $u_0^0 \in X_0$.
\Repeat
\State Set $\Xi^0_N:=1$ and $\Upsilon^0_N:=0$.
\While {$\Xi^n_N > \lambda \Upsilon^n_N$}
\State Perform a single iterative linearization step~\eqref{eq:lindisproblem} to obtain $\un{n+1}$ from $\un{n}$.
\State Update $n \gets n+1$.
\State Set~$\Xi^n_N:=\|\un{n}-\un{n-1}\|_X$, and compute the error estimator 
$\Upsilon^{n}_N:=\eta_N(\un{n})$ from~\eqref{eq:eta}.
\EndWhile
\State \multiline{Let $u_N:=u_N^n \in X_N$, and enrich the Galerkin space $X_N$ appropriately based on the error estimator $\eta_N(u_N)$ in order to obtain $X_{N+1}$.}
\State Define $u_{N+1}^0 := u_N$ by inclusion $X_{N+1} \hookleftarrow X_N$.
\State Update $N\gets N+1$, and set~$n:=0$.
\Until {$\eta_N(u_N^n) < \tol$.}\\
\Return the sequence of discrete solutions $u_N \in X_N$. 
\end{algorithmic}
\end{algorithm}

\begin{remark} \label{rem:specialcases}
We emphasize that we do not know (a priori) if the while loop of Algorithm~\ref{alg:praetal} always terminates after finitely many steps. Moreover, it may happen that $\eta_N(u_N)=0$, for some $N \geq 0$, i.e.~the algorithm terminates. Let us provide two comments on this issue:
\begin{enumerate}[(a)]
\item Suppose that there is an enrichment $X_N$ of $X_0$ generated by the above Algorithm~\ref{alg:praetal} such that $\Xi^n_N > \lambda \Upsilon^n_N$ for all $n \geq 0$; in this situation, the while loop will never end. Given the assumptions of Proposition~\ref{prop:contractionlike}, it follows from Corollary~\ref{cor:null} that $\Xi^n_N \to 0$ as $n\to\infty$. In addition, by virtue of Theorem~\ref{thm:it} (applied to the discrete setting~\eqref{eq:discreteproblem} and~\eqref{eq:lindisproblem}), we have that $\un{n} \to \dissol{N}$ as $n \to \infty$. Then, invoking the reliability (A2) and the continuity (A1), we conclude that
\begin{align*}
\norm{u^\star-\dissol{N}}_X \leq \cref{eq:a2} \eta_N(\dissol{N}) = \cref{eq:a2}\lim_{n \to \infty}  \eta_N(\un{n})
=\cref{eq:a2}\lim_{n \to \infty}  \Upsilon^n_N \leq \cref{eq:a2} \lim_{n \to \infty} \lambda^{-1} \Xi^n_N=0.
\end{align*}
It follows that $u^\star=\dissol{N}$, and therefore $\un{n} \to u^\star$ as $n \to \infty$. In particular, Algorithm~\ref{alg:praetal} will generate an approximate solution which, for sufficiently large~$n$, is arbitrarily close to the exact solution of~\eqref{eq:F=0}.
\item If, for some $n,N \in \mathbb{N}$, the while loop terminates, then we have the bound $\Xi^n_N\le\lambda\Upsilon^n_N$. Thus, in the special situation where $\Upsilon^n_N=\eta_N(u_N^n)=0$, we directly obtain that~$\Xi^n_N=0$. Then, employing Lemma~\ref{lem:discreteerrorestimate} and Remark~\ref{rem:XN}, we find that $\norm{\dissol{N}-\un{n}}_X\leq \cref{eq:discreteerrorestimate} \|\un{n}-\un{n-1}\|_X=0$, i.e. $\dissol{N}=\un{n}$. Consequently, recalling~(A2), we deduce that
\[
\|u^\star-\un{n}\|_X=\|u^\star-\dissol{N}\|_X
\le\cref{eq:a2} \eta_N(\dissol{N})
=\cref{eq:a2} \Upsilon^n_N=0.
\]
We obtain that $\un{n}=u^\star$, i.e. the exact solution of~\eqref{eq:F=0} is found.
\end{enumerate}
\end{remark}

\subsection{Perturbed contractivity}

We will now turn to the proof of the convergence of Algorithm~\ref{alg:praetal}. More precisely, we will show that the sequence $u_N$ generated by the above ILG procedure converges, under certain assumptions, to the exact solution $u^\star$ of \eqref{eq:F=0}. In view of Remark~\ref{rem:specialcases} we may assume that the while loop always terminates after finitely many steps with $\eta_N(u_N)>0$ for all $N \geq 0$. 

We begin with the following result, which corresponds to~\cite[Proposition 4.10]{GantnerHaberlPraetoriusStiftner:17}. Since we consider general Galerkin discretizations, an additional assumption, cf.~the perturbed contraction~\eqref{eq:contractionperb} below, is imposed.

\begin{proposition} \label{prop:plainconvergence} 
Let {\rm (F1)--(F2)} and {\rm (A1)} be satisfied, and~$\lambda\in(0,\cref{eq:clambda}^{-1})$ be given.  Moreover, for each~$N\ge 0$, assume that the while loop of Algorithm~\ref{alg:praetal} terminates after finitely many steps, thereby yielding an output~$u_N\in X_N$, with $\eta_N(u_N) >0$. Furthermore, suppose that there are constants $0 < \qref{eq:contractionperb} <1$ and $\cref{eq:contractionperb} >0$ such that it holds the perturbed contraction bound
\begin{equation} \label{eq:contractionperb}
 \eta_{N+1}(\dissol{N+1})^2 \leq \qref{eq:contractionperb} \eta_N(\dissol{N})^2+ \cref{eq:contractionperb} \norm{\dissol{N+1}-\dissol{N}}_X^2 \qquad \forall N \geq 0,
\end{equation}
where $\dissol{N} \in X_N$ is the unique solution of \eqref{eq:discreteproblem}. Then, we have that $\eta_N(u_N) \to 0$ as $N \to \infty$.
\end{proposition}

\begin{proof}
Set $X_\infty:=\overline{\bigcup_{N \geq 0} X_N}$, and denote by $\dissol{\infty} \in X_\infty$ the solution of the weak formulation
 \begin{align*}
   \dissol{\infty} \in X_\infty: \qquad \dprod{\F(\dissol{\infty}),v}=0 \qquad \forall v \in X_\infty.
 \end{align*}
For any $N \geq 0$, Galerkin orthogonality reads $\dprod{\F(\dissol{\infty})-\F(\dissol{N}),v}=0$ for all $v \in X_N$. Thus, by the Lipschitz continuity (F1) and strong monotonicity (F2) we find that 
 \begin{align*}
    \Fsm \norm{\dissol{\infty}-\dissol{N}}_X^2 &\leq \dprod{\F(\dissol{\infty})-\F(\dissol{N}),\dissol{\infty}-\dissol{N}}\\ 
    &= \dprod{\F(\dissol{\infty})-\F(\dissol{N}),\dissol{\infty}-v}\\
  &\leq \Flc \norm{\dissol{\infty}-\dissol{N}}_X\norm{\dissol{\infty}-v}_X,
 \end{align*}
for any $v \in X_N$. This results in the C\'{e}a type estimate 
 \begin{equation} \label{eq:ceatypeestimate}
  \norm{\dissol{\infty}-\dissol{N}}_X \leq \frac{\Flc}{\Fsm} \min_{v \in X_N} \norm{\dissol{\infty}-v}_X.
 \end{equation}
Recalling the nestedness~\eqref{eq:nested} of the Galerkin spaces, and exploiting the definition of $X_\infty$, the above bound \eqref{eq:ceatypeestimate} directly implies that $\dissol{N} \to \dissol{\infty}$ for $N \to \infty$. Consequently, we deduce that $\norm{\dissol{N+1}-\dissol{N}}_X \to 0$ for $N \to \infty$. Hence, by \eqref{eq:contractionperb}, the estimator $\eta_N(\dissol{N})^2$, $N\ge 0$, is contractive up to a non-negative perturbation which tends to~0. This implies that $\eta_N(\dissol{N}) \to 0$ as $N \to \infty$, see, e.g., \cite[Lemma 2.3]{AuradaFerrazPraetorius:12}. Since $u_N=u_N^n$ satisfies $\norm{\un{n}-\un{n-1}}_X \leq \lambda \eta_N(\un{n})$ by construction of Algorithm~\ref{alg:praetal}, Lemma~\ref{lem:equivalentestimators} yields the equivalence of $\eta_N(\dissol{N})$ and $\eta_N(u_N)$. Hence, we conclude that $\eta_N(u_N) \to 0$ as $N \to \infty$.
\end{proof}

\begin{remark}
We emphasize that the perturbed contraction property~\eqref{eq:contractionperb} is satisfied, for instance, in the special case of the finite element method; see~\cite{{GantnerHaberlPraetoriusStiftner:17}} for details.
\end{remark}

Combining the above Proposition~\ref{prop:plainconvergence} and Lemma~\ref{lem:equivalentestimators} leads to the following result.

\begin{cor} 
Given the same assumptions as in Proposition~\ref{prop:plainconvergence} and, additionally, {\rm (A2)}, then $u_N \to u^\star$ for $N \to \infty$, where the sequence~$\{u_N\}_{N\ge 0}$ is generated by the Algorithm~\ref{alg:praetal}, and $u^\star$ is the unique solution of \eqref{eq:F=0}.
\end{cor}

\begin{proof}
Let~$u_N=\un{n}\in X_N$, $n\ge 1$, be the output of Algorithm~\ref{alg:praetal} based on the Galerkin space~$X_N$. Then, by virtue of~\eqref{eq:unusbound}, and due to~(A2) and~\eqref{eq:eq}, we have
\begin{equation}\label{eq:aux20190408a}
\|u^\star-u_N\|_X
\le\|\dissol{N}-\un{n}\|_X+\|u^\star-\dissol{N}\|_X
\le\lambda\cref{eq:discreteerrorestimate}\eta_N(\un{n})+\cref{eq:a2}\eta_N(\dissol{N})
\le\cref{eq:cc}\eta_N(u_N),
\end{equation}
with
\begin{equation}\label{eq:cc}
\cref{eq:cc}:=\lambda\cref{eq:discreteerrorestimate}+\cref{eq:a2}(1+\lambda\cref{eq:clambda}).
\end{equation}
Applying Proposition~\ref{prop:plainconvergence} completes the proof.
\end{proof}

\subsection{Linear convergence}

In this section we show the \emph{linear} convergence of the output sequence $\{u_N\}_{N \geq 0}$ generated by Algorithm~\ref{alg:praetal}. Our analysis follows closely the work~\cite[Theorem~5.3]{GantnerHaberlPraetoriusStiftner:17}. Again, we formulate and prove the result within a more general setting, and, for this purpose, under the additional assumption~\eqref{eq:contractionperb} as before.

Letting
\begin{equation}\label{eq:gamma}
\gamma:=\frac{\Fsm}{2 \cref{eq:contractionperb}}>0,
\end{equation}
with~$\nu>0$ from~(F2), we introduce the quantity
\begin{equation}\label{eq:DeltaN}
\Delta_N:=\H(\dissol{N})-\H(u^\star)+\gamma \eta_N(\dissol{N})^2,
\end{equation}
where~$u^\star\in X$ and~$\dissol{N}\in X_N$ are the (unique) solution of~\eqref{eq:F=0} and its Galerkin approximation from~\eqref{eq:discreteproblem}, respectively. 
By virtue of Lemma~\ref{lemma:energydifference}, provided that~{\rm(F1)--(F3)} hold, we observe that
\[
\Delta_N\ge\frac{\Fsm}{2}\norm{\dissol{N}-u^\star}^2_X\ge 0,
\]
for any~$N\ge 0$.

\begin{theorem} 
Let $\F$ satisfy {\rm (F1)--(F3)}, and assume {\rm (A1)--(A2)}. Furthermore, suppose that there are constants $0 < \qref{eq:contractionperb} <1$ and $\cref{eq:contractionperb} >0$ such that~\eqref{eq:contractionperb} holds true. Then, upon setting 
\begin{equation}\label{eq:qDelta}
\qref{eq:contractionperb}<\qref{eq:qDelta}:=\frac{\Flc \cref{eq:a2}^2+2\gamma\qref{eq:contractionperb}}{\Flc \cref{eq:a2}^2+2\gamma}<1,
\end{equation}
with~$\gamma$ from~\eqref{eq:gamma}, and with~$\lambda\in(0,\cref{eq:clambda}^{-1})$, the following contraction property holds: If the while loop of Algorithm~\ref{alg:praetal} terminates after finitely many steps with $\eta_N(u_N) >0$, for all $N \geq 0$, then we have the (linear) contraction property 
\begin{equation}\label{eq:deltacontraction}
  \Delta_{N+1} \leq \qref{eq:qDelta} \Delta_{N} \qquad \forall N \geq 0.
\end{equation}
Moreover, there exists a constant $\cref{eq:linearconvergence}>0$ such that
\begin{equation} \label{eq:linearconvergence}
 \eta_{N+K}(u_{N+K})^2 \leq \cref{eq:linearconvergence} \qref{eq:qDelta}^K \eta_N(u_N)^2 \qquad \forall N,K \geq 0,
\end{equation}
i.e. the error estimators decay at a linear rate.
\end{theorem}

\begin{proof}
Given any integers~$N\ge K\ge 0$, and corresponding Galerkin subspaces~$X_K \subseteq X_N\subset X$. Then, using Lemma~\ref{lemma:energydifference}, 
with~$X$ being replaced by~$X_N$, we have that
 \begin{equation} \label{eq:energykn}
    \frac{\Fsm}{2} \norm{\dissol{K}-\dissol{N}}_X^2 \leq \H(\dissol{K})-\H(\dissol{N}) \leq \frac{\Flc}{2} \norm{\dissol{K}-\dissol{N}}^2_X.
 \end{equation}
Furthermore, recalling~\eqref{eq:DeltaN}, and using the perturbed contraction assumption~\eqref{eq:contractionperb}, we obtain that
\begin{align*}
 \Delta_{N+1}
 & \leq \H(\dissol{N})-\H(u^\star)-(\H(\dissol{N})-\H(\dissol{N+1}))+\gamma\left(\qref{eq:contractionperb} \eta_N(\dissol{N})^2+\cref{eq:contractionperb} \norm{\dissol{N+1}-\dissol{N}}_X^2\right).
\end{align*}
Invoking \eqref{eq:energykn}, and applying the definition of $\gamma$ from~\eqref{eq:gamma}, we arrive at
\begin{align*}
 \Delta_{N+1} 
 &\leq \H(\dissol{N})-\H(u^\star) +\left(\gamma \cref{eq:contractionperb}-\frac{\Fsm}{2}\right) \norm{\dissol{N+1}-\dissol{N}}_X^2+\gamma \qref{eq:contractionperb} \eta_N(\dissol{N})^2\\
&= \H(\dissol{N})-\H(u^\star)+\gamma \qref{eq:contractionperb} \eta_N(\dissol{N})^2.
\end{align*}
Here, owing to the reliability assumption~(A2), and upon implementing~\eqref{eq:energydifference}, we note that
\begin{align*}
\qref{eq:contractionperb} \eta_N(\dissol{N})^2
&=\qref{eq:qDelta} \eta_N(\dissol{N})^2-(\qref{eq:qDelta}-\qref{eq:contractionperb}) \eta_N(\dissol{N})^2\\
&\le\qref{eq:qDelta} \eta_N(\dissol{N})^2
-(\qref{eq:qDelta}-\qref{eq:contractionperb})\cref{eq:a2}^{-2}\norm{u^\star-\dissol{N}}_X^2\\
&\le\qref{eq:qDelta} \eta_N(\dissol{N})^2
-2(\qref{eq:qDelta}-\qref{eq:contractionperb})\Flc^{-1}\cref{eq:a2}^{-2}(\H(\dissol{N})-\H(u^\star)).
\end{align*}
Thus, it follows that
\[
 \Delta_{N+1} \leq\left(1-2\gamma(\qref{eq:qDelta}-\qref{eq:contractionperb})\Flc^{-1}\cref{eq:a2}^{-2}\right)(\H(\dissol{N})-\H(u^\star))+\gamma\qref{eq:qDelta} \eta_N(\dissol{N})^2.
\]
Noticing that $1-2\gamma(\qref{eq:qDelta}-\qref{eq:contractionperb})\Flc^{-1}\cref{eq:a2}^{-2}
=\qref{eq:qDelta}$, yields~\eqref{eq:deltacontraction}. Hence, by induction, it holds that $\Delta_{N+K} \leq \qref{eq:qDelta}^K \Delta_N$. From this inequality, together with \eqref{eq:eq} and the fact that $\H(\dissol{N+K})-\H(u^\star)\ge 0$, cf.~\eqref{eq:energydifference}, we conclude that
\begin{align*}
 \eta_{N+K}(u_{N+K})^2 \simeq \eta_{N+K}(\dissol{N+K})^2 \leq \gamma^{-1}\Delta_{N+K} \leq \gamma^{-1}\qref{eq:qDelta}^K \Delta_N. 
\end{align*}
Employing again~\eqref{eq:energydifference} and making use of the reliability condition~(A2), this leads to
\begin{align*}
 \eta_{N+K}(u_{N+K})^2 \lesssim \qref{eq:qDelta}^K\left(\frac{\Flc}{2 \gamma} \norm{u^\star-\dissol{N}}_X^2 + \eta_N(\dissol{N})^2\right) \lesssim \qref{eq:qDelta}^K \eta_N(\dissol{N})^2.
\end{align*}
Therefore, once again applying~\eqref{eq:eq}, we deduce~\eqref{eq:linearconvergence}.
\end{proof}

\subsection{Uniform bound for the number of linearization steps}

Given that the while loop of Algorithm~\ref{alg:praetal} terminates after finitely many steps for all $N \geq 0$, we will show that the number of iterative linearization steps~\eqref{eq:lindisproblem} on each Galerkin space $X_N$, which will be denoted by $\# \mathrm{It}(N)$, can be (uniformly) bounded. As before we assume that $\eta_N(u_N)>0$ for all $N \geq 0$. The following result is in particular~\cite[Proposition~4.6]{GantnerHaberlPraetoriusStiftner:17}, however, we will adapt the proof to the effect that it applies for the contraction-like property~\eqref{eq:ndiff1diff} (instead of the contraction of the Zarantonello iteration exploited in~\cite{GantnerHaberlPraetoriusStiftner:17}). 


\begin{proposition} \label{prop:bounditerates} 
Suppose {\rm (F1)--(F4)} and {\rm (A1)--(A2)}. Let $\lambda\in(0,\cref{eq:clambda}^{-1})$ be the adaptivity parameter from Algorithm~\ref{alg:praetal}. Suppose that the while loop of Algorithm~\ref{alg:praetal} terminates after finitely many steps with $\eta_N(u_N)>0$ for all $N \geq 0$. Then, the number of iterative linearization steps $\# \mathrm{It}(N)$ on $X_N$ satisfies the estimate 
\begin{align} \label{eq:bounditerates}
 \# \mathrm{It}(N) \leq \frac{2}{\log\left(1+\cref{eq:sumbound}^{-1}\right)}\log\left(\left(C\lambda^{-1}+C'\right)\left(1+\cref{eq:sumbound}^{-1}\right)\max\left\{1,\frac{\eta_{N-1}(u_{N-1})}{\eta_N(u_N)}\right\}\right)+1, 
 \end{align}
for all $N\ge 1$, where 
\[
C:=\frac{\Flc^{\nicefrac12}\cref{eq:sumbound}^{\nicefrac12}\cref{eq:cc}}{\sqrt2C^{\nicefrac12}_\H} 
  \qquad \text{and} \qquad C':= \frac{C\cref{eq:a1} \cref{eq:sumbound}^{\nicefrac{1}{2}}}{1-\left(1+\cref{eq:sumbound}^{-1}\right)^{-\nicefrac{1}{2}}}.
\]
\end{proposition}

\begin{proof}
We split the proof into two parts.

\emph{Part 1:}
Let 
\[
d_N:=\frac{\eta_{N-1}(u_{N-1})}{\eta_N(u_N)},\qquad N\ge 1,
\] 
and choose~$n\in\mathbb{N}$, $n \geq 2$, minimal such that
\[
 0<\frac{C d_N \left(1+\cref{eq:sumbound}^{-1}\right)^{1-\nicefrac{n}{2}}}{1-C' d_N \left(1+\cref{eq:sumbound}^{-1}\right)^{1-\nicefrac{n}{2}}} \leq \lambda.
\]
We denote by $k:=\# \mathrm{It}(N)$ the number of linearization steps on $X_N$, and aim to show that $k \leq n$. To this end, we assume by contradiction that $k>n > 1$. By definition of Algorithm~\ref{alg:praetal}, $k \geq 1$ is the minimal number such that $\norm{\un{k}-\un{k-1}}_X \leq \lambda \eta_N(\un{k})$. Invoking \eqref{eq:ndiff1diff}, we find that
\[
  \norm{\un{n}-\un{n-1}}_X^2 \leq \cref{eq:sumbound} \left(1+\cref{eq:sumbound}^{-1}\right)^{2-n} \norm{\un{1}-\un{0}}_X^2.
\]
Let us bound $\norm{\un{1}-\un{0}}_X^2$. From (F4) and \eqref{eq:energydifference} we obtain
\begin{align*}
 C_\H \norm{u_N^1-u_N^0}_X^2 \leq \H(u_N^0)-\H(u_N^1) \leq \H(u_N^0)-\H(u^\star) \leq \frac{\Flc}{2} \norm{u^\star-u_N^0}_X^2.
\end{align*} 
Hence, we infer the bound
 \begin{align*}
  \norm{\un{n}-\un{n-1}}_X^2 &\leq\frac{\Flc\cref{eq:sumbound}}{2C_\H} \left(1+\cref{eq:sumbound}^{-1}\right)^{2-n}  \norm{u^\star-\un{0}}_X^2.
 \end{align*}
Since $u_N^0=u_{N-1}$, we apply~\eqref{eq:aux20190408a} to deduce that 
 \[
 \norm{u^\star-\un{0}}_X
 \leq\cref{eq:cc} \eta_{N-1}(u_{N-1}).
 \]
Therefore, we arrive at 
\begin{align}\label{eq:bounditeq1} 
\begin{split}
\norm{\un{n}-\un{n-1}}_X^2 
&\leq\frac{\Flc\cref{eq:sumbound}\cref{eq:cc}^2}{2C_\H} \left(1+\cref{eq:sumbound}^{-1}\right)^{2-n}\eta_{N-1}(u_{N-1})^2
=C^2 \left(1+\cref{eq:sumbound}^{-1}\right)^{2-n}d^2_N\eta_{N}(u_{N})^2.
\end{split}
\end{align}
Invoking the stability (A1), we find that
\begin{align*}
\eta_N(u_N)=\eta_N(\un{k}) \leq \eta_N(\un{n})+\cref{eq:a1} \norm{\un{k}-\un{n}}_X \leq \eta_N(\un{n})+\cref{eq:a1} \sum_{j=n}^{k-1} \norm{\un{j+1}-\un{j}}_X. 
\end{align*}
Restricting~\eqref{eq:ndiff1diff} to~$X_N$ and shifting indices, yields that 
\begin{align*}
 \norm{\un{j+1}-\un{j}}_X \leq \cref{eq:sumbound}^{\nicefrac{1}{2}}  \left(\left(1+\cref{eq:sumbound}^{-1}\right)^{-\nicefrac{1}{2}}\right)^{j-n}\norm{\un{n}-\un{n-1}}_X, \qquad j \ge n .
\end{align*}
Combining the previous inequalities leads to
\begin{align*}
 \eta_N(u_N) &\leq \eta_N(\un{n})+\cref{eq:a1} \cref{eq:sumbound}^{\nicefrac{1}{2}}\norm{\un{n}-\un{n-1}}_X \sum_{j=0}^{k-n-1} \left(\left(1+\cref{eq:sumbound}^{-1}\right)^{-\nicefrac{1}{2}}\right)^{j} \\
 & \leq \eta_N(\un{n})+  \frac{\cref{eq:a1} \cref{eq:sumbound}^{\nicefrac{1}{2}}}{1-\left(1+\cref{eq:sumbound}^{-1}\right)^{-\nicefrac{1}{2}}}\norm{\un{n}-\un{n-1}}_X.
\end{align*}
Inserting this inequality into \eqref{eq:bounditeq1} yields
\begin{align*}
 \norm{\un{n}-\un{n-1}}_X 
 &\leq Cd_N\left(1+\cref{eq:sumbound}^{-1}\right)^{1-\nicefrac{n}{2}}\eta_{N}(\un{n})
 +C'd_N\left(1+\cref{eq:sumbound}^{-1}\right)^{1-\nicefrac{n}{2}}\norm{\un{n}-\un{n-1}}_X.
\end{align*}
A straightforward manipulation leads to 
\begin{align*}
 \norm{\un{n}-\un{n-1}}_X \leq \frac{C d_N \left(1+\cref{eq:sumbound}^{-1}\right)^{1-\nicefrac{n}{2}}}{1-C' d_N \left(1+\cref{eq:sumbound}^{-1}\right)^{1-\nicefrac{n}{2}}}\eta_N(\un{n}) \leq \lambda \eta_N(\un{n}),
\end{align*}
which contradicts the minimality of $k$, wherefore it must hold that $k \leq n$.

\emph{Part 2:}
Set 
\[
M :=\frac{2}{\log\left(1+\cref{eq:sumbound}^{-1}\right)}\log\left(\left(C\lambda^{-1}+C'\right)\left(1+\cref{eq:sumbound}^{-1}\right)\max\left\{1,d_N\right\}\right).
\] 
We show that
\begin{equation}\label{eq:aux20190408b}
0<\frac{C d_N \left(1+\cref{eq:sumbound}^{-1}\right)^{1-\nicefrac{\ceil{M}}{2}}}{1-C' d_N \left(1+\cref{eq:sumbound}^{-1}\right)^{1-\nicefrac{\ceil{M}}{2}}}\le\lambda,
\end{equation}
where~$\ceil{M}$ denotes the smallest integer larger than or equal to~$M$. We verify that
\begin{align}
 \left(1+\cref{eq:sumbound}^{-1}\right)^{\nicefrac{\ceil{M}}{2}}
 &\ge\exp\left(\nicefrac{M}{2}\log\left(1+\cref{eq:sumbound}^{-1}\right)\right)\nonumber \\
 &= \left(C\lambda^{-1}+C'\right)\left(1+\cref{eq:sumbound}^{-1}\right)\max\{1,d_N\}\label{eq:aux20190424a}  \\
 & > C' d_N \left(1+\cref{eq:sumbound}^{-1}\right)>0.\nonumber
\end{align}
This proves the lower bound in~\eqref{eq:aux20190408b}. In a similar way, the upper bound is obtained. Indeed, multiplying~\eqref{eq:aux20190424a} by~$\lambda$, we have that
\[
 \lambda  \left(1+\cref{eq:sumbound}^{-1}\right)^{\nicefrac{\ceil{M}}{2}} \geq \lambda C' d_N \left(1+\cref{eq:sumbound}^{-1}\right)+C d_N \left(1+\cref{eq:sumbound}^{-1}\right).
\]
Thence, a simple manipulation leads to 
\[\lambda \left(1-C' d_N \left(1+\cref{eq:sumbound}^{-1}\right)^{1-\nicefrac{\ceil{M}}{2}}\right) \geq C d_N \left(1+\cref{eq:sumbound}^{-1}\right)^{1-\nicefrac{\ceil{M}}{2}},\]
which immediately implies the claim. In summary, using Part~1 of the proof, we conclude that $\# \mathrm{It}(N)=k \leq\ceil{M}$.  
\end{proof}

\begin{remark}
 We note that Proposition~\ref{prop:bounditerates} does not assert a uniform bound for the number of linearization steps since the right-hand side of \eqref{eq:bounditerates} depends on $N$. For many sensible error estimators, however, it holds the contraction property
 \begin{align} \label{eq:qest}
  \norm{u^\star-\dissol{N}}_X \simeq \qref{eq:qest} \norm{u^\star-\dissol{N-1}}_X,
\end{align}
for $N$ large enough, with a contraction constant $0<\qref{eq:qest}<1$. Furthermore, we may assume a reliability and efficiency estimate:
\begin{align} \label{eq:sensestimator}
 \norm{u^\star-\dissol{N}}_X \simeq \cref{eq:sensestimator} \eta_N(\dissol{N}), \qquad N \geq 0.
\end{align}
Then, combining \eqref{eq:qest}, \eqref{eq:sensestimator}, and \eqref{eq:eq}, we obtain
$\eta_N(u_N) \simeq \eta_{N-1}(u_{N-1})$, and consequently, a uniform bound on the number of linearization steps on each Galerkin subspace $X_N$ is guaranteed.
\end{remark}

\begin{remark}
As was mentioned earlier, this result is one of the key parts in the computational complexity analysis of the ILG Algorithm~\ref{alg:praetal}. Indeed, following along the lines of \cite[\S6]{GantnerHaberlPraetoriusStiftner:17} and replacing \cite[Proposition~4.6]{GantnerHaberlPraetoriusStiftner:17} by the above Proposition~\ref{prop:bounditerates}, the almost optimal computational work of Algorithm~\ref{alg:praetal}, under suitable assumptions, can be established in the context of finite element method discretizations.
\end{remark}


\section{Numerical experiments} \label{sec:examples}

In this section we test our ILG Algorithm~\ref{alg:praetal} with two numerical experiments in the context of finite element discretizations of stationary conservation laws.

\subsection{Model problem}

On an open, bounded and polygonal domain $\Omega \subset \mathbb{R}^2$, with Lipschitz boundary $\Gamma=\partial \Omega$, let us consider the second-order elliptic partial differential equation
\begin{align} \label{eq:operatorscl}
u\in X:\qquad \F(u):= - \nabla \cdot \left[\mu\left(\left|\nabla u\right|^2\right) \nabla{u}\right]-g=0\qquad\text{in }X^\star.
\end{align}
Here, we choose~$X:=H^1_0(\Omega)$ to be the standard Sobolev space of $H^1$-functions on~$\Omega$ with zero trace along~$\Gamma$; the inner product and norm on~$X$ are defined, respectively, by~$(u,v)_X:=(\nabla u,\nabla v)_{L^2(\Omega)}$ and~$\norm{u}_X:=\|\nabla u\|_{L^2(\Omega)}$, for $u,v\in X$.  We suppose that $g \in X^\star=H^{-1}(\Omega)$ in~\eqref{eq:operatorscl} is given, and the diffusion parameter $\mu \in C^1([0,\infty))$ fulfils the monotonicity property
\begin{align} \label{en:assmu}
m_\mu(t-s) \leq \mu(t^2)t-\mu(s^2)s \leq M_\mu (t-s), \qquad t \geq s \geq 0,
\end{align}
with constants $M_\mu\ge m_\mu>0$. Under this condition the nonlinear operator~$\F:\,H^1_0(\Omega)\to H^{-1}(\Omega)$ from~\eqref{eq:operatorscl} can be shown to satisfy~(F1) and~(F2), with $\nu=m_{\mu}$ and $\Flc=3M_{\mu}$;
see~\cite[Proposition~25.26]{Zeidler:90}. Moreover, $\F$ has a potential $\operator{H}:X \to  \mathbb{R}$ given by
\begin{align*}
\operator{H}(u):=\int_\Omega \psi\left(\left|\nabla u\right|^2\right) \dx-\dprod{g,u},\qquad u\in X,
\end{align*}
where~$\psi(s):=\nicefrac12\int_0^s \mu(t) \dt$, $s\ge 0$. 
The weak form of the boundary value problem~\eqref{eq:operatorscl} in $X$ reads:
\begin{align}\label{eq:sclweak}
u\in X:\qquad \int_\Omega \muf{u} \nabla u \cdot \nabla v \dx = \dprod{g,v} \qquad \forall v \in X.
\end{align}

In \cite[Section 5.1]{HeidWihler:18} the convergence of the Zarantonello, Ka\v{c}anov, and Newton iteration for the nonlinear boundary value problem~\eqref{eq:operatorscl} was examined.

\subsection{Discretization and refinement indicator}

For the sake of discretizing \eqref{eq:sclweak}, and thereby of obtaining an ILG formulation for~\eqref{eq:operatorscl}, we will use a conforming finite element framework. We consider a sequence of hierarchical, regular and shape-regular meshes $\{\mathcal{T}_N\}_{N\ge 1}$ that partition the domain~$\Omega$ into open and disjoint triangles~$T \in\mathcal{T}_N$ such that $\overline{\Omega}=\bigcup_{T \in \mathcal{T}_N} T$.
Moreover, we consider the finite element space 
\[
X_N:=\left\{v \in H^1_0(\Omega): v|_T \in \mathcal{P}_1(T) \ \forall T \in \mathcal{T}_N\right\},
\]
where we signify by $\mathcal{P}_1(T)$ the space of all affine functions on $T \in \mathcal{T}_N$. The mesh refinements in Algorithm~\ref{alg:praetal} are obtained by means of the newest vertex bisection and the D\"{o}rfler marking strategy, see \cite{Mitchell:91} and~\cite{Doerfler:96}, respectively.

For an edge $e \subset \partial T^+ \cap \partial T^{-}$, which is the intersection of (the closures of) two neighbouring elements $T^{\pm} \in \mathcal{T}_N$, we signify by $\jmp{\bm v}|_e=\bm{v}^{+}|_e \cdot \bm{n}_{T^+}+\bm{v}^{-}|_e\cdot \bm{n}_{T^{-}}$ the jump of a (vector-valued) function $\bm{v}$ along~$e$, where $\bm{v}^{\pm}|_e$ denote the traces of the function $\bm{v}$ on the edge $e$ taken from the interior of $T^{\pm}$, respectively, and $\bm{n}_{T^{\pm}}$ are the unit outward normal vectors on $\partial T^{\pm}$, respectively. For $u \in X_N$ we define the local refinement indicator, for each $T \in \mathcal{T}_N$, and the global error indicator, respectively, by 
\[
 \eta_N(T,u)^2:=h_T^2 \twon{g}{T}^2+h_T \twon{\muf{u}\nabla u}{\partial T \setminus \Gamma}^2, \qquad \eta_N(u):=\left(\sum_{T \in \mathcal{T}_N}\eta_N(T,u)^2 \right)^{\nicefrac{1}{2}}. 
\]
This error estimator satisfies the assumptions (A1)--(A2) for the problem under consideration; we refer to \cite[\S8.3]{GantnerHaberlPraetoriusStiftner:17} for details.

\subsection{Experiments}

We revisit two experiments from \cite{HeidWihler:18}, whereby we test the (modified) adaptive ILG Algorithm~\ref{alg:praetal}. We consider the L-shaped domain $\Omega=(-1,1)^2 \setminus ([0,1] \times [-1,0])$, and start the computations with an initial mesh consisting of 192 uniform triangles. The procedure is run until the number of elements exceeds $10^6$. Moreover, we will always choose the initial guess $u_0^0 \equiv 0$.

\subsubsection{Smooth solution} \label{sec:smoothsolution} We consider the nonlinear diffusion coefficient $\mu(t)=(t+1)^{-1}+\nicefrac{1}{2}$,
for $t \geq 0$, and select~$g$ in~\eqref{eq:operatorscl} such that the analytical solution of \eqref{eq:sclweak} is given by the smooth function $u^\star(x,y)=\sin(\pi x) \sin(\pi y)$. 
It is straightforward to verify that $\mu$ fulfils the bounds~\eqref{en:assmu} so that the assumptions (F1)--(F3) are satisfied. In addition, the convergence of the \emph{Zarantonello}~\eqref{eq:zarantonelloit}, \emph{Ka\v{c}anov}~\eqref{eq:kacanovstrong}, and \emph{Newton}~\eqref{eq:newtonstrong} iterative linearization procedures are guaranteed (with the parameter $\dpa=0.85$ and~$\dpa=1$ in case of the Zarantonello and Newton method, respectively); see~\cite{HeidWihler:18}. 
A priori, for the Newton method, we remark that choosing the damping parameter~$\dpa=1$ (potentially resulting in quadratic convergence of the iterative linearization close to the solution) might lead to a divergent iteration for the given boundary value problem; for this reason, a prediction and correction strategy, which guarantees convergence (and which does not cause any correction of the damping parameter in the current experiments), is presented in~\cite[Remark 2.8]{HeidWihler:18}. 

In Figure~\ref{fig:smoothconvergence} we plot the error estimators (solid lines) and true errors (dashed lines) of our three linearization schemes against the number of elements $|\mathcal{T}_N|$ in the triangulation. In addition, the dashed line without any markers is the graph of the function $|\mathcal{T}_N|^{-\nicefrac{1}{2}}$. We observe the optimal convergence rate $\mathcal{O}(|\mathcal{T}_N|^{-\nicefrac{1}{2}})$ for both (almost) uniform and adaptive mesh refinements corresponding to the parameters $\theta_*=0$ and $\theta_*=0.5$, respectively, in the D\"{o}rfler marking strategy, see~\cite[\S4.2, Eq.~(M$_*$)]{Doerfler:96}.

\begin{figure} 
 \subfloat[$\lambda=0.5$ and $\theta_*=0.5$]{\includegraphics[width=0.48\textwidth]{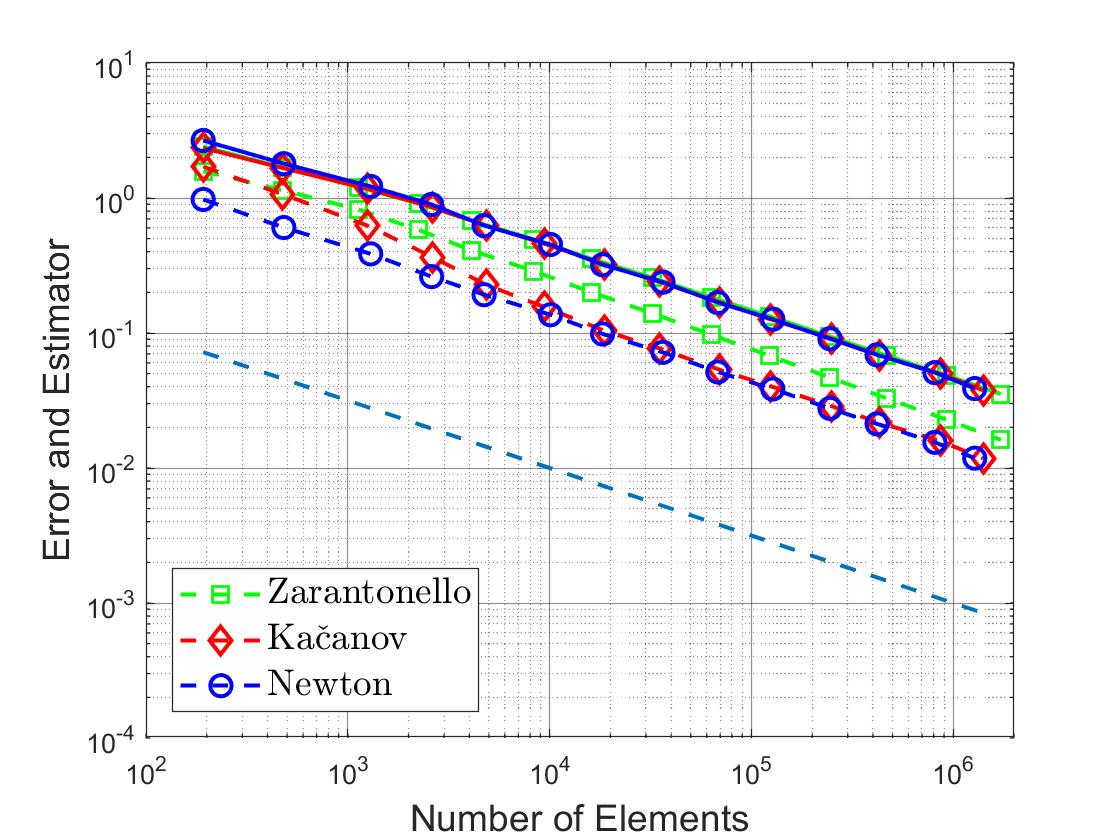}}\hfill
 \subfloat[$\lambda=0.5$ and $\theta_*=0$]{\includegraphics[width=0.48\textwidth]{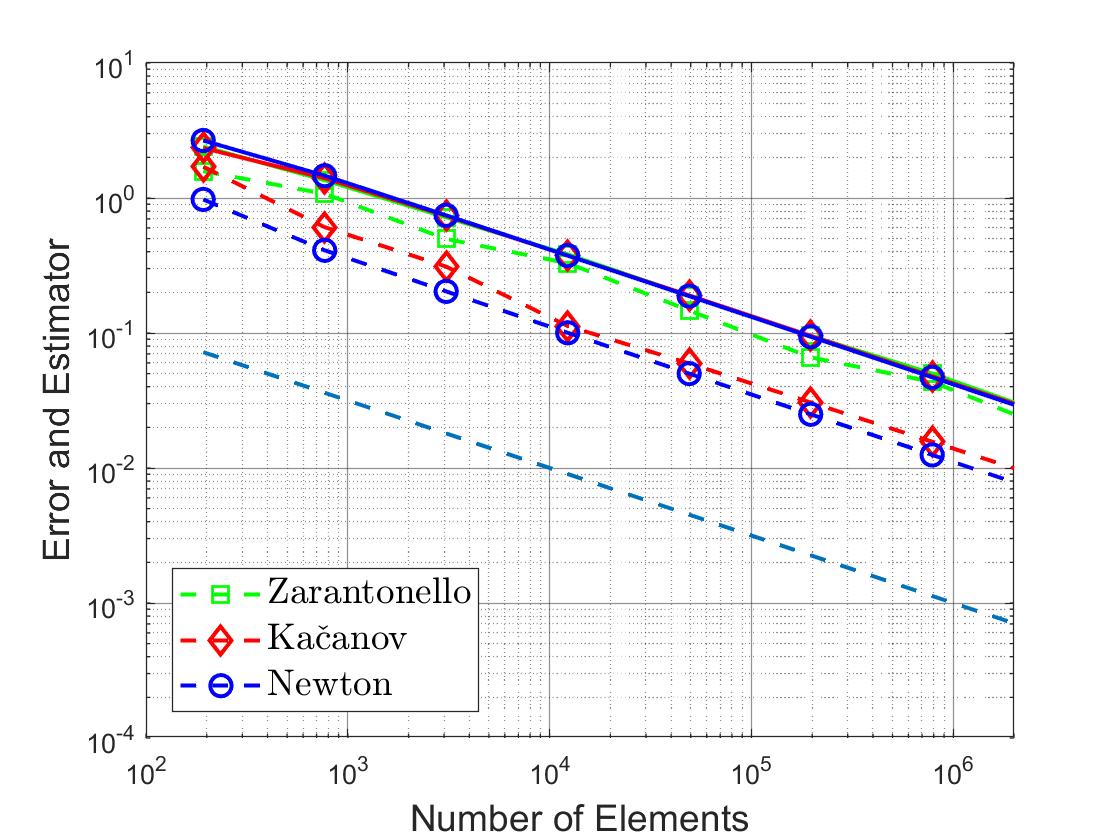}}
 \caption{Experiment~\ref{sec:smoothsolution}: Convergence rates. Left: Adaptively refined meshes. Right: (Almost) uniform meshes.}\label{fig:smoothconvergence}
\end{figure}

In Figure~\ref{fig:smoothitnum} we can observe the uniformly bounded number of linearization steps on a given mesh for any of the three considered fixed-point methods, and for both choices of the adaptivity parameter $\lambda=0.1$ and $\lambda=0.001$ in Algorithm~\ref{alg:praetal}. For the value $\lambda=0.1$, the number of linearization steps on a given mesh does not differ essentially between the three considered iteration schemes. However, there is a remarkable difference for the choice $\lambda=0.001$. The Newton iteration clearly outperforms the other two fixed-point methods, which is not surprising because of the local quadratic convergence regime. In addition, we can also observe that the Ka\v{c}anov method is superior to the Zarantonello iteration in view of the number of linearization steps.

\begin{figure}
 \subfloat[$\lambda=0.1$ and $\theta_*=0.5$]{\includegraphics[width=0.48\textwidth]{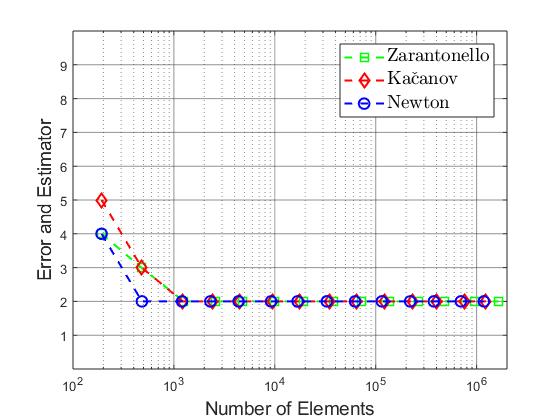}}\hfill
 \subfloat[$\lambda=0.001$ and $\theta_*=0.5$]{\includegraphics[width=0.48\textwidth]{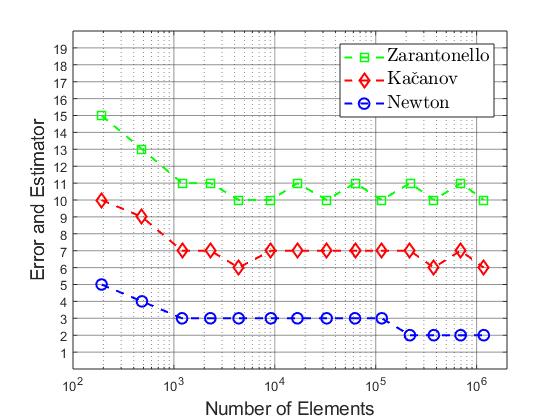}}
 \caption{Experiment~\ref{sec:smoothsolution}: Number of iterations.}  \label{fig:smoothitnum}
\end{figure}

\subsubsection{Nonsmooth solution} \label{sec:nonsmoothsolution} In our second experiment, we consider the nonlinear diffusion parameter $\mu(t)=1+\mathrm{e}^{-t}$, for $t \geq 0$. Again, it is easily seen that~$\mu$ satisfies~\eqref{en:assmu}. Moreover, it can be shown that the three linearization schemes under consideration will converge for appropriate choices of the parameter $\dpa$, see~\cite{HeidWihler:18}. We choose $g$ in \eqref{eq:operatorscl} such that the analytical solution is given by 
\[
u^\star(r,\varphi)=r^{\nicefrac{2}{3}}\sin\left(\nicefrac{2\varphi}{3}\right)(1-r \cos(\varphi))(1+r \cos(\varphi))(1- r \sin(\varphi))(1+r \sin(\varphi))\cos(\varphi),
\]
where $r$ and $\varphi$ are polar coordinates. This is the prototype singularity for (linear) second-order elliptic problems with homogeneous Dirichlet boundary conditions in the L-shaped domain; in particular, we note that the gradient of~$u^\star$ is unbounded at the origin. We let $\dpa=0.5$ for the Zarantonello iteration, and use the damping parameter $\dpa=1$ for the Newton method as in the experiment before.

For the choice $\theta_*=0.5$ in D\"orfler's marking procedure we retain the (almost) optimal convergence rate for both the error and the estimator. Due to the singularity, however, the convergence rate is reduced when the mesh is (almost) uniformly refined (i.e.~corresponding to the value~$\theta_*=0$). For the number of linearization steps on each Galerkin space, we can make the same observations as for the smooth case from before.

\begin{figure} 
 \subfloat[$\lambda=0.5$ and $\theta_*=0.5$]{\includegraphics[width=0.48\textwidth]{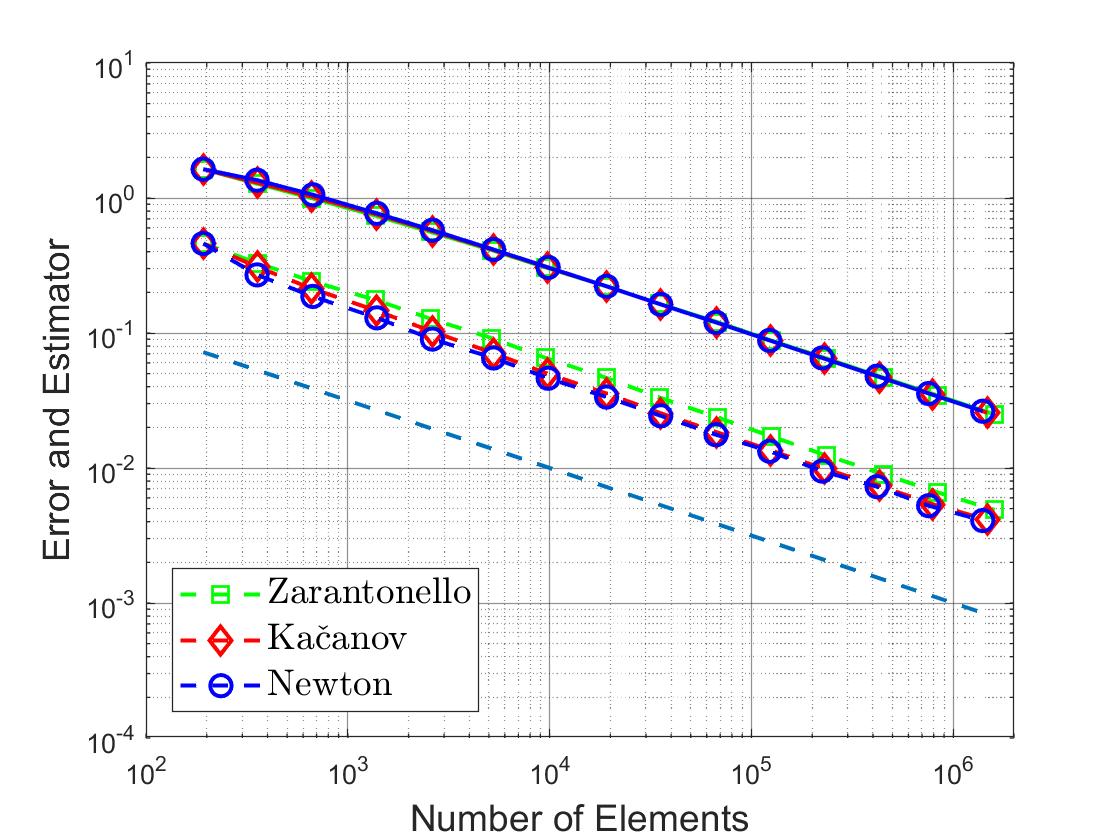}}\hfill
 \subfloat[$\lambda=0.5$ and $\theta_*=0$]{\includegraphics[width=0.48\textwidth]{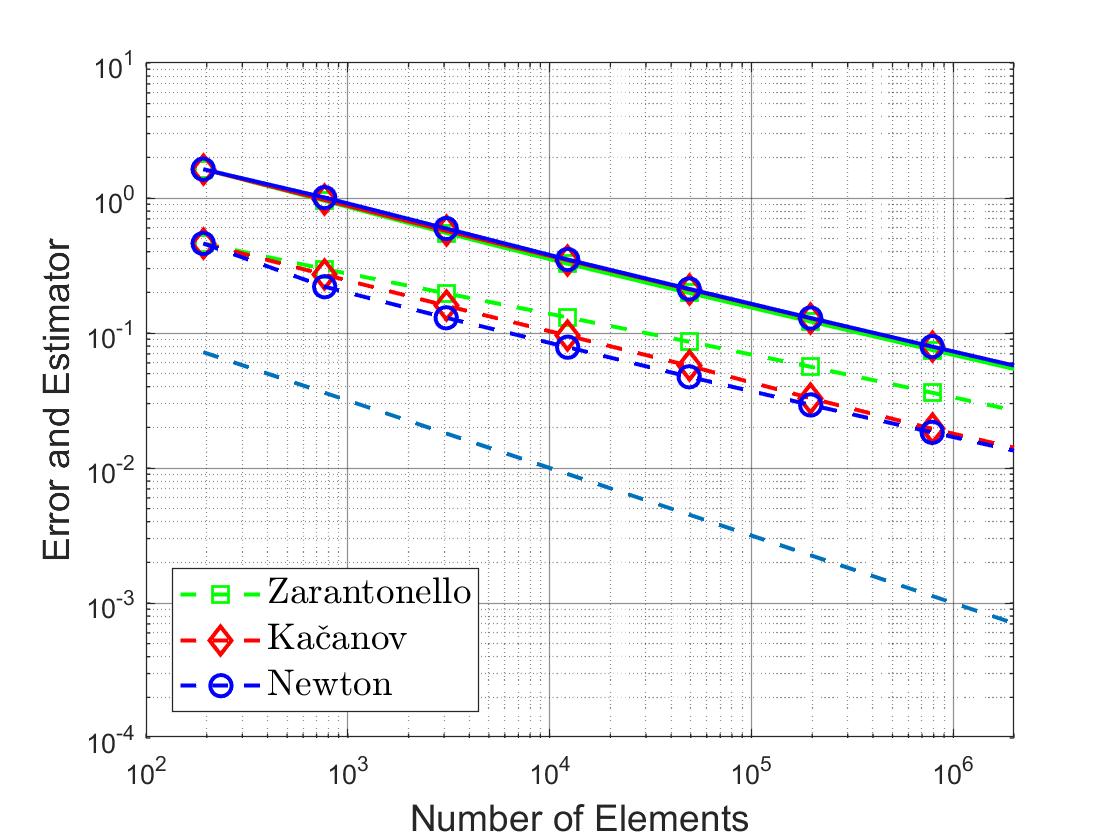}}
 \caption{Experiment~\ref{sec:nonsmoothsolution}: Convergence rates. Left: Adaptively refined meshes. Right: (Almost) uniform meshes.}
\end{figure}

\begin{figure} 
 \subfloat[$\lambda=0.1$ and $\theta_*=0.5$]{\includegraphics[width=0.48\textwidth]{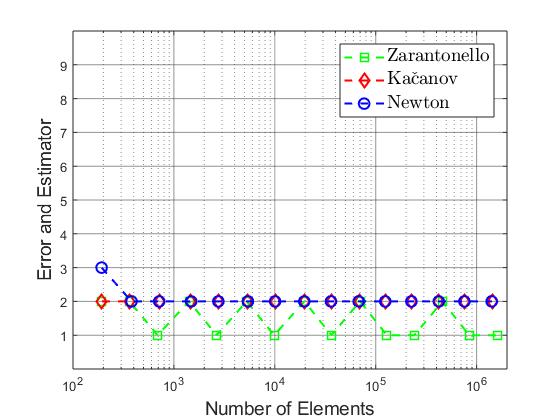}}\hfill
 \subfloat[$\lambda=0.001$ and $\theta_*=0.5$]{\includegraphics[width=0.48\textwidth]{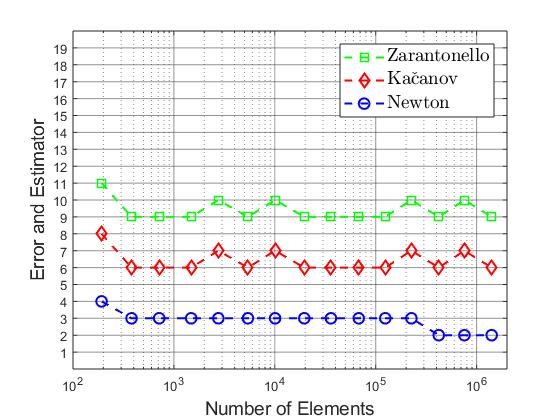}}
 \caption{Experiment~\ref{sec:nonsmoothsolution}: Number of iterations. } 
\end{figure}

\section{Conclusions}\label{sec:conclude}

We have established a contraction-like property of the unified iteration scheme \eqref{eq:fp1}, which is key for the convergence analysis of the adaptive ILG Algorithm~\ref{alg:praetal}. In particular, we were able to generalize some of the results from~\cite{GantnerHaberlPraetoriusStiftner:17} including the linear convergence of the general ILG procedure and the (uniform) boundedness of the number of linearization steps on each Galerkin space. We underline that the latter property constitutes an important stepping stone for the analysis of optimal computational complexity, cf.~\cite[\S6 and \S7]{GantnerHaberlPraetoriusStiftner:17}.

\bibliographystyle{amsplain}
\bibliography{references}
\end{document}